\documentclass[12pt, twoside]{article}
\usepackage{amsmath,amsthm,amssymb}
\usepackage{times}
\usepackage{enumerate}
\usepackage{extarrows}
\usepackage{color}
\usepackage{tabularx}
\usepackage{multirow}
\usepackage{graphicx}

\def\titlerunning#1{\gdef\titrun{#1}}
\makeatletter
\def\author#1{\gdef\autrun{\def\and{\unskip, }#1}\gdef\@author{#1}}
\def\address#1{{\def\and{\\\hspace*{18pt}}\renewcommand{\thefootnote}{}
\footnote {#1}}
\markboth{\autrun}{\titrun}}
\makeatother
\def\email#1{e-mail: #1}

\newtheorem{theorem}{Theorem}[section]
\newtheorem{lemma}[theorem]{Lemma}
\newtheorem{definition}[theorem]{Definition}
\newtheorem{proposition}[theorem]{Proposition}
\newtheorem{remark}[theorem]{Remark}
\newtheorem{corollary}[theorem]{Corollary}
\newtheorem{question}[theorem]{Question}
\newtheorem{example}[theorem]{Example}

\newcommand{\R}{\mathbb{R}}

\newcommand{\Proof}{\begin{proof}}
\newcommand{\End}{\end{proof}}

\numberwithin{equation}{section}

\newcommand{\PreserveBackslash}[1]{\let\temp=\\#1\let\\=\temp}
\newcolumntype{C}[1]{>{\PreserveBackslash\centering}p{#1}}
\newcolumntype{R}[1]{>{\PreserveBackslash\raggedleft}p{#1}}
\newcolumntype{L}[1]{>{\PreserveBackslash\raggedright}p{#1}}

\newcolumntype{I}{!{\vrule width 1pt}}
\newlength\savedwidth

\begin{document}


\baselineskip=15pt


\titlerunning{Asymptotic orbits in contact systems}

\title{Variational construction of connecting orbits between Legendrian graphs}

\author{Liang Jin \, Jun Yan \,and\, Kai Zhao}

\date{\today}

\maketitle

\address{Jin Liang: Department of Mathematics, Nanjing University of Science and Technology, Nanjing 210094, China; Fakult\"{a}t f\"{u}r Mathematik, Ruhr-Universit\"{a}t Bochum, Universit\"{a}tstra$\beta$e 150, D-44801 Bochum, Germany\\
\email{jl@njust.edu.cn; Liang.Jin@ruhr-uni-bochum.de}
  \and Jun Yan: School of Mathematical Sciences, Fudan University, Shanghai 200433, China;
  \email{yanjun@fudan.edu.cn}
  \and
  Kai Zhao:  School of Mathematical Sciences, Fudan University, Shanghai 200433, China;
  \email{zhao$\_$kai@fudan.edu.cn}}

\begin{abstract}
  Motivated by the problem of global stability of thermodynamical equilibria in non-equilibrium thermodynamics formulated in a recent paper \cite{EP}, we introduce some mechanisms for constructing semi-infinite orbits of contact Hamiltonian systems connecting two Legendrian graphs from the viewpoint of Aubry-Mather theory and weak KAM theory.
\end{abstract}

%


\section{Introduction}
\setcounter{equation}{0}
\setcounter{footnote}{0}

Let $M$ be a closed, connected smooth Riemannian manifold and $\Sigma=J^1 M=T^{\ast}M\times\R$ the manifold of 1-jets of functions on $M$. We use either $\sigma$ or the coordinates $(q,p,u)$ to denote points in $\Sigma$, where $(q,p)$ is the usual coordinates on $T^{\ast}M$ and $u\in\R$. The kernel of the Gibbs $1$-form $\alpha=du-pdq$ defines the standard contact structure $\xi\subset T\Sigma$, which makes $(\Sigma,\xi)$ into a canonical contact manifold. In the following context,
\begin{itemize}
  \item $\pi^u:\Sigma\rightarrow T^\ast M;\quad\sigma=(q,p,u)\mapsto(q,p)$ denotes the projection forgetting $u$-component,
  \item $\pi_q:\Sigma\rightarrow M;\quad\sigma=(q,p,u)\mapsto q$ denotes the projection onto $q$-component.
\end{itemize}

\subsection{Contact structure and classical thermodynamics}
According to V.I.Arnold \cite{Ar-Contact}, ``the first person who understood the significance of contact geometry for physics and thermodynamics'' is J.W.Gibbs. As is well known, see \cite[Part III, Chapter 1]{Zo-Book} and the references therein, Gibbs laid the foundation of classical thermodynamics by using only the functions of thermodynamic state, such as entropy and energy, as coordinates to describe the thermodynamic process, and then give the mathematical formulation of two principles of classical thermodynamics with the help of the 1-form $\alpha$: any equilibrium process corresponds to an oriented path $\gamma$ in $\Sigma$ and takes place in such a way that $\dot{\gamma}\in\ker(\alpha)=\xi$. From Gibbs' formulation, one naturally thinks of the sets of equilibrium states of a thermodynamic system as integral manifolds of the contact structure $\xi$, especially

\begin{definition}[Legendrian submanifolds]\label{lm}
An integral submanifold $\Lambda\subset\Sigma$ of $\xi$ whose dimension is maximal, i.e., dim $\Lambda=$ dim $M$, is called Legendrian. Furthermore, $\Lambda$ is called a Legendrian graph if $\pi_q|_{\Lambda}:\Lambda\rightarrow M$ is a diffeomorphism.
\end{definition}
It is necessary to

\begin{remark}\cite[Lecture 2, Theorem 4]{Ar-PDE}
Legendrian graphs coincide with 1-graphs of functions: for every Legendrian graph $\Lambda$, there is a function $u\in C^1(M,\R)$ such that
\[
\Lambda=\Lambda_{u}:=\{\,(q,d_q u(q), u(q)):\,q\in M\,\},
\]
and we shall use the notation $\Lambda_u$ if we want to emphasis the generating function.
\end{remark}

Therefore, $(\Sigma,\xi)$ serves as the phase space in the geometrical description of classical (or \textbf{equilibrium}) thermodynamics. After the fundamental works of Gibbs, the classical (or equilibrium) thermodynamics, which is the theory of properties of matter in a state of thermodynamics equilibrium and mainly deals with the reversible evolution of equilibrium states, became the study of contact geometry of the phase space.

\subsection{Contact Hamiltonian systems and non-equilibrium thermodynamics}
Unfortunately, systems found in nature are rarely in thermodynamic equilibrium, mainly due to the exchange of matter and energy with the environment and the chemical reactions inside the system. Does this mean that Gibbs' framework is nonsense to such systems? The discoveries of 20 century's thermodynamics probably give an answer from negative direction.

\vspace{0.5em}
For instance, physical experiments show that when a thermodynamic system in an equilibrium state undergoes a perturbation, it probably moves to a non-equilibrium state and then enter an interesting \textbf{relaxation process} that driving the system gradually returns to the \textbf{original} equilibrium. The past 30 years have witnessed a trend to interpret this relaxation processes via contact Hamiltonian flows, an analogy of Hamiltonian flow on contact manifolds \cite{B,EP,G,Rajeev,S}. One reason to choose such flows comes from the fact that they are transformations of $\Sigma$ preserving the contact structure $\xi$ and thus sets of equilibrium states (means that the flow transforms Legendrian submanifolds into Legendrian submanifolds) as well. As a consequence, the generating vector field $X$ satisfies the equation
\begin{equation}\label{contact trans}
\mathcal{L}_{X}\alpha=f\alpha,
\end{equation}
where $\mathcal{L}_{X}$ denotes the Lie derivative along the vector field $X$ and $f$ a nowhere vanishing function on $\Sigma$. It turns out that $X$ is uniquely determined by the function $H=i_X\alpha$, called \textbf{contact Hamiltonian}, on $\Sigma$, where $i_X$ denotes the inner product of a vector field and a 1-form. Set $\dot{q}=i_{X}dq,\,\dot{p}=i_{X}dp,\,\dot{u}=i_{X}du$, the equation \eqref{contact trans} reads in the coordinates $(q,p,u)$ as the \textbf{contact Hamiltonian system}
\begin{equation}\label{ch}
X:
\begin{cases}
\dot{q}=\frac{\partial H}{\partial p}(q,p,u),\\
\dot{p}=-\frac{\partial H}{\partial q}(q,p,u)-\frac{\partial H}{\partial u}(q,p,u)p,\\
\dot{u}=\frac{\partial H}{\partial p}(q,p,u)\cdot p-H(q,p,u).
\end{cases}
\end{equation}
In the following context, we shall use the notation $X_H$ instead of $X$ to emphasis the role of $H$, and the corresponding phase flow is denoted by $\varphi^t_H$. In particular, if the Hamiltonian is independent of $u$, then \eqref{ch} reduces to the classical Hamiltonian system.

\vspace{1em}
Once and for all, we assume there is a Legendrian submanifold
\begin{itemize}
  \item  $\Lambda_-\subset\Sigma$ coincides with the pre-assigned set of original equilibrium states.
\end{itemize}
The ingredients of the geometric model of relaxation process are included in the following
\begin{definition}\label{non-equilibrium}
For $H\in C^{\infty}(\Sigma,\R)$, the pair $(H,\Lambda_-)$ is called a non-equilibrium thermodynamic system (a system for short) if $H$ generates a \textbf{complete} phase flow $\varphi^{t}_H$ and $H|_{\Lambda_-}\equiv0$. It follows that $\Lambda_-$ is an invariant manifold under $\varphi^t_H$.
\end{definition}

\begin{remark}
In the language of non-equilibrium thermodynamics, the phase flow $\varphi^t_H$ represents the thermodynamic process, which can be chosen by determining the contact Hamiltonian $H$ \textbf{from different physical considerations}. However, the above definition indicates the choice of $H$ can not be arbitrary and should guarantee that the process preserves the set of original equilibrium states.
\end{remark}

\subsection{Prigogine's question on the global stability of thermodynamic equilibrium}
In his 1977 Nobel lecture \cite[Page 269]{Prigogine-Nobel}, Ilya.Prigogine attributed the local stability of thermodynamic equilibrium in the relaxation processes to the fact that $u$ (called thermodynamic potential) serves as a Lyapunov function near the equilibrium and then raise the global question:
\begin{itemize}
  \item \textbf{Can we extrapolate this stability property further away from equilibrium?}
\end{itemize}
This note is motivated by the latest work \cite{EP} of M.Entov and L.Polterovich, in which the authors reformulate Prigogine's question into a question concerning contact Hamiltonian dynamics:
\begin{question}\label{global stability}\cite[Section 3, Question 3.1]{EP}
Given a system $(H,\Lambda_-)$ and a subset $\Sigma_0\subset\Sigma$ of the phase space, does there exist an initial condition $\sigma_0\in\Sigma_0$ whose trajectory in the thermodynamical process generated by $H$ asymptotically converges to the equilibrium submanifold $\Lambda_-$?
\end{question}
In \cite{EP}, the authors also offered their answer to Question \ref{global stability} for the case that $\Sigma_0=\Lambda_0$ is a Legendrian submanifold under the assumption proposed by physicists working in non-equilibrium thermodynamics:
\begin{itemize}
  \item [\textbf{(H1)}] $\Lambda_-$ is a local attractor, i.e., there is a neighborhood $\mathcal{O}_-$ of $\Lambda_-$ in $\Sigma$ such that for any $\sigma\in\mathcal{O}_-$, the orbit $\{\varphi^t_H\sigma\,:\,t\geqslant0\}\subset\mathcal{O}_-$ and the set $\omega(\sigma)$ is a non-empty subset of $\Lambda_-$.
\end{itemize}
In fact, they constructed semi-infinte orbits of $\varphi^t_H$ starting from $\Lambda_0$ and converges to $\Lambda_-$ when $t$ goes to infinity. The methods used there is based on an existence mechanism, for finite time-length trajectories of \eqref{ch} between Legendrian submanifolds, called \textbf{interlinking} established from Legendrian Contact Homology in hard contact geometry \cite{EP0,EP1}.

\vspace{1em}
As is indicated in the last section, Hamiltonian system could be seen as a special case of \eqref{ch}. Topics related to understanding the chaotic behavior of orbits in Hamiltonian system lie at the centre in this field. For instance, one of the main goal of the celebrated Aubry-Mather theory \cite{M1,M2} is to answer

\vspace{1em}
\textit{``\cite{M2} whether there exists an orbit which in the infinite past tends to one region of phase space and in the infinite future tends to another region of phase space''} and \textit{``the possibility of finding an orbit which visits a prescribed sequence of regions of phase space in turn''}.

\vspace{1em}
Thus the construction of semi-infinite orbits asymptotic to certain invariant sets plays the role of building block for studying Hamiltonian dynamics from the above viewpoint. This fact also suggests the importance of Question 1.5 in the study of contact Hamiltonian dynamics.

\subsection{Main results}
In recent years, the study of contact Hamiltonian system from the viewpoint of Aubry-Mather theory \cite{M1,M2} and weak KAM theory \cite{Fathi_book} has fruitful consequences, including the proof of vanishing discount limit \cite{DFIZ} raised in the homogenization problem of Hamilton-Jacobi equations \cite{LPV}, and attracts much interests. In their ground breaking paper \cite{MS}, the authors focus on developing Aubry-Mather theory for certain contact Hamiltonian called conformally symplectic, i.e., $H(q,p,u)=\lambda u+h(q,p)$ with $\lambda>0$ being a constant. They successfully apply their results to investigate the global dissipative dynamics of the system.

\vspace{1em}
In the series of works \cite{WWY1}-\cite{WWY3}, the authors found an implicitly defined variational principle for general contact Hamiltonian systems, and use it to built the variational theory in the spirit of \cite{Fathi_book} and \cite{M1,M2}. In this note, we concentrate on the application of this theory to problems with more dynamical ingredients. Precisely, the aim of this note is to provide our answers to Question \ref{global stability} by methods developed in particular in \cite{WWY1}-\cite{WWY3}. With slight abuse of notation, we use $|\cdot|_q$ to denote the norm induced on the cotangent space $T^{\ast}_{q}M$. In the following context, we shall restrict ourselves to consider the case when
\begin{itemize}
  \item $\Lambda_-,\Lambda_0$ are Legendrian graphs, i.e., there is $u_-,u_0\in C^\infty(M,\R)$ such that
        \[
        \Lambda_-=\Lambda_{u_-},\quad \Lambda_0=\Lambda_{u_0}.
        \]
        It follows from Definition \ref{non-equilibrium} that $u_-$ is a smooth solution to the equation
        \begin{equation}\label{HJs}\tag{HJs}
        H(q,d_q u(q),u(q))=0,\quad\text{for any}\,\,q\in M.
        \end{equation}
\end{itemize}
and a non-equilibrium thermodynamics system $(H,\Lambda_-)$ with the Hamiltonian satisfying
\begin{itemize}
  \item [\textbf{(H2)}] $\frac{\partial^2 H}{\partial p^2}(q,p,u)$ is positive definite for every $(q,p,u)\in\Sigma$ and for every $(q,u)\in M\times\R$,
        \[
        \lim_{|p|_q\rightarrow+\infty}\frac{H(q,p,u)}{|p|_q}  \rightarrow+\infty.
        \]

\end{itemize}
Now we introduce the following
\begin{definition}\label{sub-super-intro}
A \textbf{sub\,(resp. super)-deformation} of $u_-$ is a function $V\in C^\infty(M\times[0,1],\R)$ such that
\begin{itemize}
  \item $V(\cdot,1)=u_-$,
  \item $H|_{\Lambda_{V(\cdot,s)}}<0\,(\text{resp.}\,>0)\quad\text{for all}\,\,s\in[0,1)$.
\end{itemize}
The name comes from the fact that for $s\in[0,1), V(\cdot,s)$ is a strict subsolution (resp. supersolution) to the equation \eqref{HJs}, see Definition \ref{vis} in Section 2.
\end{definition}

For $\sigma\in\Sigma, \omega(\sigma)$ denotes the omega-limit set of $\sigma$ under $\varphi^t_H$ (in general maybe empty!). The first result concerns the construction of semi-infinite orbits connecting $\Lambda_0$ to $\Lambda_-$ by using Definition \ref{sub-super-intro}.
\begin{theorem}\label{main1}
Let $(H,\Lambda_-)$ be a system with $H\in C^\infty(\Sigma,\R)$ satisfies \textbf{(H2)}. Assume one of the following conditions
\begin{enumerate}[$(a)$]
  \item there is a sub-deformation $V:M\times[0,1]\rightarrow\R$ of $u_-$ such that $V(\cdot,0)\leq u_0\leq u_-$,
  \item there is a super-deformation $V:M\times[0,1]\rightarrow\R$ of $u_-$ such that $V(\cdot,0)\geq u_0\geq u_-$,
  \item there are   a sub-deformation $\underline V:M\times[0,1]\rightarrow\R$ and a super-deformation $\overline V:M\times[0,1]\rightarrow\R$ of $u_-$ such that
   $$
  \underline  V(\cdot ,0 ) \leqslant \min\{ u_0, u_-\}, \quad \max\{ u_0, u_-\}   \leqslant \overline V(\cdot ,0 ) .
   $$
\end{enumerate}
then it follows that
\begin{enumerate}[(1)]
  \item $\Lambda_-\subset\overline{\cup_{t\geqslant0}\varphi^t_H(\Lambda_0)}$,
  \item and there is $\sigma_0\in\Lambda_0$ such that $\omega(\sigma_0)\subset\Lambda_-$.
\end{enumerate}
\end{theorem}

\begin{remark}
Conclusion (1) means that every point on the set of equilibria can be approximated by some finite time-length trajectories of the thermodynamic process initiating from $\Lambda_0$.
\end{remark}

If the local stability of $\Lambda_-$ is assumed, then the conditions (a) and (b) can be replaced by some condition depending only on $0$-jets. This is included in
\begin{theorem}\label{main2}
Let $(H,\Lambda_-)$ be a system with $H\in C^\infty(\Sigma,\R)$ satisfies \textbf{(H1)}-\textbf{(H2)}. Assume one of the following conditions
\begin{enumerate}[(a)]
  \item[$(a')$] there is a continuous function $V:M\times[0,1]\rightarrow\R$ such that $V(\cdot ,0)\in C^\infty(M,\R)$ with $H|_{\Lambda_{V(\cdot,0)}}<0$ and
        \[
        H(q,d_q u_-(q),V(q,s))<0,  \quad \forall q\in M \quad   \text{and}\quad V(\cdot,0) \leq u_0\leq u_-=V(\cdot,1),
        \]

  \item[$(b')$] there is a continuous function $V:M\times[0,1]\rightarrow\R$ such that $V(\cdot ,0)\in C^\infty(M,\R)$ with $H|_{\Lambda_{V(\cdot,0)}}>0$ and
      \[
      H(q,d_q u_-(q),V(q,s)) >0,  \quad \forall q\in M \quad   \text{and}\quad V(\cdot,0) \geq u_0\geq u_-=V(\cdot,1),
      \]

  \item [$(c')$] there are continuous functions $\underline{V}, \overline{V}:M\times[0,1]\rightarrow\R$ such that $\underline{V}(\cdot,0), \overline{V}(\cdot,0)\in C^\infty(M,\R)$ with $H|_{\Lambda_{\underline{V}(\cdot,0)}}<0, H|_{\Lambda_{\overline{V}(\cdot,0)}}>0$ and
        \begin{align*}
        & H(q,d_q u_-(q),\underline{V}(q,s))<0,\quad\quad H(q,d_q u_-(q),\overline{V}(q,s))>0,\\
        &\underline{V}(q,0)\leqslant\min\{u_0(q), u_-(q)\}, \quad\,\max\{u_0(q), u_-(q)\}  \leqslant\overline{V}(q,0).
        \end{align*}
\end{enumerate}
is satisfied, then there is $\sigma_0\in\Lambda_0$ such that $\omega(\sigma_0)\subset\Lambda_-$.
\end{theorem}

\begin{remark}
The conditions $(a)-(c)$, $(a')-(c')$ listed above are stated in a homotopy flavor. They give examples, in our informal opinion, of ``weak'' version of interlink property employed in \cite{EP} that are more easy to verify directly on the contact Hamiltonian.
\end{remark}

\subsection{Organization of the paper}
The remaining of this paper is organized as follows. In Section 2, we briefly recall some necessary tools from \cite{WWY1}-\cite{WWY3} and give an extension of characteristic theory, which is crucial in our proof of Theorem \ref{main1}. Section 3 is devoted to the construction of semi-infinite orbit asymptotically converges to $\Lambda_-$ when the solution semigroup associated to the evolutionary Hamilton-Jacobi equation converges. In Section 4, our homotopy criteria are verified to guarantee the convergence of the solution semigroup with initial data $u_0$. We also illustrate our results on some examples, including some generalizations of those from \cite{EP}, in the last section.

\section{Global characteristics theory via variational methods}
To extend the characteristic theory to the global setting, we recall the variational methods developed for evolutionary Hamilton-Jacobi equation (including viscosity solutions theory) and the associated contact Hamiltonian system. A global version of characteristic theory is obtained by showing that viscosity solutions propagate along action minimizing orbits of \eqref{ch}. Notice that many objects discussed in this section is non-smooth in the classical viewpoint, thus necessary smoothness to guarantee the validity of definitions and theorems is presented in an accurate way.

\vspace{1em}
The classical characteristics theory connects the \textbf{local solvability} of the Cauchy problem of the evolutionary Hamilton-Jacobi equation
\begin{equation}\label{HJe}\tag{HJe}
\begin{cases}
\partial_t U+H(q,\partial_q U,U)=0,\quad\, (q,t)\in M\times(0,+\infty),\\
\hspace{5.4em}U(\cdot,0)=v,\hspace{2.65em}q\in M,
\end{cases}
\end{equation}
to the study of contact Hamiltonian system \eqref{ch} near $\Lambda_v$, here $v\in C^2(M)$ is a smooth initial data. More precisely, if one assume $U:M\times[0,+\infty)\rightarrow\R$ is a $C^2$ solution to \eqref{HJe}, then every trajectory $\varphi^t_H \sigma:=(q(t),p(t),u(t)), t\geq0$ of \eqref{ch}, called \textbf{characteristic}, starting from $\sigma\in\Lambda_{v}$ satisfies the identities
\begin{equation}\label{ch-sol}
p(t)=\partial_q U(q(t),t),\quad u(t)=U(q(t),t).
\end{equation}
Equivalently, for every $t\geq0, \{\varphi^t_H\sigma:\sigma\in\Lambda_v\}=\Lambda_{U(\cdot,t)}$. Following this spirit, one arrives at

\begin{theorem}\cite[Lecture 2, Theorem 3]{Ar-PDE} or \cite[Chapter 3, Theorem 2]{Evans-PDE}\label{local-chm}
For any $v\in C^2(M,\R)$, there are $\delta>0$ and a solution $U\in C^2(M\times[0,\delta],\R)$ to \eqref{HJe}, so that
\begin{enumerate}[(1)]
  \item for any $t\in[0,\delta], \pi_q\circ\varphi^t_H:\Lambda_v\rightarrow M$ is a diffeomorphism,
  \item any characteristic segment $\varphi^t_H\sigma=(q(t),p(t),u(t)), \sigma\in\Lambda_v, t\in[0,\delta]$ satisfies the identities \eqref{ch-sol}.
\end{enumerate}
\end{theorem}
Notice that in general, given a smooth initial data $v$, Theorem \ref{local-chm} only allow us to construct smooth solution to \eqref{HJe} locally. The reason comes from the fact that, after the projection by $\pi_q$, the characteristics starting from $\Lambda_v$ may intersect at some large $t$. Thus even for smooth initial data $v$, there does not exist a global solution $U\in C^2(M\times[0,\infty),\R)$ to \eqref{HJe}. To construct solutions to \eqref{HJe}, it is necessary to extend the notion of `solutions' to include non-smooth functions. The right one, namely \textbf{viscosity solution}, was firstly introduced by M.Crandall and P.L.Lions in \cite{CL}, and is now widely accepted as the natural framework for the theory of Hamilton-Jacobi equations and certain second order PDEs.

\begin{definition}\label{vis}
A continuous function $u:M\rightarrow\R$ is called a viscosity subsolution (resp. supersolution) of \eqref{HJs} if for any $q\in M$ and $\phi\in C^{1}(M,\R)$ such that $u-\phi$ attains a local maximum (resp. minimum) at $q$,
\begin{equation}\label{sub-super}
H(q,d_{q}\phi(q),\phi(q))\leq(\text{resp.}\geq)\,\,0;
\end{equation}
$u$ is called a viscosity solution if it is both a viscosity sub and supersolution of \eqref{HJs}. Moreover, a viscosity subsolution is said to be \textbf{strict} if the inequality $\leqslant$ \eqref{sub-super} is replaced by $<$ at any $q\in M$.

\vspace{1em}
A continuous function $U: M\times[0,\infty)\rightarrow\R$ is called a viscosity subsolution (resp. supersolution) of \eqref{HJe} if $U(\cdot,0)\leq(\text{resp.}\geq)\,\,v$ on $M$ and for any $(q,t)\in M\times(0,+\infty)$ and $\Phi$ a $C^{1}$ function defined on a neighborhood of $(q,t)$ such that $U-\Phi$ attains a local maximum (resp. minimum) at $(q,t)$,
\[
\partial_{t}\Phi(q,t)+H(q,\partial_{q}\Phi(q,t),U(q,t))\leq(\text{resp.}\geq)\,\,0;
\]
$U$ is called a viscosity solution if it is both a viscosity sub and supersolution of \eqref{HJe}.
\end{definition}
From now on, solutions to \eqref{HJs} and \eqref{HJe} are always understood in the viscosity sense. Now we begin to give a brief summary of results in \cite{WWY1}-\cite{WWY3} concerning the variational part of the theory of viscosity solutions to \eqref{HJe} and \eqref{HJs}.

\subsection{A variational principle associated to \eqref{HJe}}
Let $TM$ denote the tangent bundle of $M$. A point of $TM$ will be denoted by $(q,\dot{q})$, where $q\in M$ and $\dot{q}\in T_qM$. Recall that $p\in T^{\ast}_{q}M$ is a linear form on $T_{q}M$, we use $\langle\cdot,\cdot\rangle$ to denote the canonical pairing between tangent and cotangent bundle. For a contact Hamiltonian $H\in C^3(\Sigma,\R)$ satisfying \textbf{(H2)}, we define the corresponding Lagrangian $L:TM\times\R\rightarrow\R$ by
\[
L(q,\dot{q},u)=\sup_{p \in T_q^{\ast}M}\{\langle p,\dot{q}\rangle-H(q,p,u)\},
\]
i.e., $L$ is the convex dual of $H$ with respect to $p$. The following action functions provide a formulation of the variational principle defined by the equation \eqref{HJe}. Notice that the action function is \textbf{implicitly} defined since $H$ depends on the $u$-variable. In \cite{CCWY}\cite{CCJWY}, the authors show that Hoglotz' variational principle also is a effective tool.
\begin{proposition}\cite[Theorem 2.1, 2.2]{WWY3}\label{Implicit variational}
Given any $(q_0,u_0)\in M\times\R$, there exist two continuous functions $h_{q_0,u_0}(q,t)$ and $h^{q_0,u_0}(q,t)$ called the \textbf{backward} and \textbf{forward} action function respectively, defined on $M\times (0,+\infty)$ by
\begin{align}\label{eq:Implicit1}
h_{q_0,u_0}(q,t)=&\inf_{\substack{\gamma(t)=q\\ \gamma(0)=q_0 } }\Big\{u_0+\int_0^t L(\gamma(\tau), \dot \gamma(\tau),h_{q_0,u_0}(\gamma(\tau) ,\tau )  )\ d\tau\Big\},\\ \label{eq:Implicit2}
h^{q_0,u_0}(q,t)=&\sup_{\substack{\gamma(t)=q_0\\ \gamma(0)=q } }\Big\{u_0-\int_0^t L(\gamma(\tau), \dot \gamma(\tau),h^{q_0,u_0}(\gamma(\tau) ,t-\tau )  )\ d\tau\Big\},
\end{align}
where the infimum and supremum are taken among Lipschitz continuous curves $\gamma:[0,t]\rightarrow M$ and are achieved. Moreover, if $\gamma_1$ and $\gamma_2$ achieve the infimum in \eqref{eq:Implicit1} and supremum in \eqref{eq:Implicit2} respectively, then $\gamma_1,\gamma_2\in C^1([0,t],M)$. Set
\begin{align*}
&q_1(\tau):=\gamma_1(\tau), \quad u_1(\tau):=h_{q_0,u_0}(\gamma_1(\tau),\tau), \hspace{3em} p_1(\tau):=\frac{\partial L}{\partial \dot q}(\gamma_1(\tau),\dot{\gamma}_1(\tau),u_1(\tau)),\\
&q_2(\tau):=\gamma_2(\tau), \quad u_2(\tau):=h^{q_0,u_0}(\gamma_2(\tau),t-\tau)), \quad p_2(\tau):=\frac{\partial L}{\partial \dot q}(\gamma_2(\tau),\dot{\gamma}_2(\tau),u_2(\tau)),
\end{align*}
then $(q_1(\tau),p_1(\tau),u_1(\tau))$ and $(q_2(\tau),p_2(\tau),u_2(\tau))$ satisfy \eqref{ch} with
\begin{align*}
&q_1(0)=q_0,\quad q_1(t)=q,\quad \lim_{\tau\to 0^+ }u_1(\tau)=u_0,\\
&q_2(0)=q,\quad q_2(t)=q_0,\quad \lim_{\tau\to t^-}u_2(\tau)=u_0.
\end{align*}
\end{proposition}

As a direct consequence of Proposition \ref{Implicit variational}, we obtain
\begin{corollary}\label{Minimality}
Given $q_0,q\in M, u_0\in\R$ and $t>0$, set $(q(\tau),p(\tau),u(\tau))=\varphi^\tau_H(q(0),p(0),u(0))$ and
\[
S^{q,t}_{q_0,u_0}=\big\{(q(\tau),p(\tau),u(\tau)), \tau\in [0,t]\,:\, q(0)=q_0,q(t)=q,u(0)=u_0\big\},
\]
\[
S^{q_0,u_0}_{q,t}=\big\{(q(\tau),p(\tau),u(\tau)), \tau\in [0,t]\,:\,q(0)=q, q(t)=q_0, u(t)=u_0\big\},
\]
then for any $ (q,t)\in M\times(0,+\infty)$,
\begin{equation}\label{eq:inf}
h_{q_0,u_0}(q,t)=\inf\,\{u(t):(q(\tau),p(\tau),u(\tau))\in S^{q,t}_{q_0,u_0}\},
\end{equation}
\begin{equation}
h^{q_0,u_0}(q,t)=\sup \{u(0):(q(\tau),p(\tau),u(\tau))\in S^{q_0,u_0}_{q,t}\}.
\end{equation}
\end{corollary}

\vspace{0.5em}
We collect some fundamental properties of the action functions here, which are frequently used in the later context. For details and proofs of these properties, we refer to the paper \cite{WWY2}.
\begin{proposition}{\cite{WWY2}}\label{fundamental-prop}
The backward and forward action functions satisfy
\begin{enumerate}
    \item[(1)]\textbf{($u_0$-monotonicity)} Given $q_0\in M, u_1<u_2\in\R$,\,\,for all $(q,t)\in M\times(0,\infty)$,
    \[
	h_{q_0,u_1}(q,t)< h_{q_0,u_2}(q,t),\quad h^{q_0,u_1}(q,t)< h^{q_0,u_2}(q,t).
    \]

	\item[(2)]\textbf{(Markov property)} Given $(x_0,u_0)\in M\times\R$,\,\,for all $t,\tau>0$ and $q\in M$,
	\begin{equation}\label{markov}
    \begin{split}
	&h_{q_0,u_0}(q, t+\tau)=\inf_{q_1\in M}h_{q_1,h_{q_0,u_0}(q_1,t)}(q,\tau),\\
    &h^{q_0,u_0}(q, t+\tau)=\sup_{q_1\in M}h^{q_1,h^{q_0,u_0}(q_1,t)}(q,\tau).
    \end{split}
	\end{equation}
	Moreover, the infimum is attained at $q_1$ if and only if there exists a $C^1$ minimizer $\gamma$ of $h_{q_0,u_0}(q,t+\tau)$ with $ \gamma(t)=q_1$, the supremum is attained at $q_1$ if and only if there exists a $C^1$ minimizer $\gamma$ of $h^{q_0,u_0}(q,t+\tau)$ with $ \gamma(t)=q_1$.

	\item[(3)]\textbf{(Lipschitz continuity)} The functions
    \[
    (q_0,u_0,q,t)\mapsto h_{q_0,u_0}(q,t),\quad (q_0,u_0,q,t)\mapsto h^{q_0,u_0}(q,t)
    \]
    are locally Lipschitz continuous on the domain $M\times\R\times M\times(0,+\infty)$.
\end{enumerate}
\end{proposition}

It turns out that any solution to \eqref{HJe} can be expressed by action functions. The representation involves some families of nonlinear operators which we now introduce.

\begin{definition}[Solution semigroups]\label{semi-group}
For each $v\in C(M,\R)$ and $(q,t)\in M\times(0,+\infty)$, define
\begin{equation}\label{eq:Tt-+ rep}
\begin{split}
T^{-}_t v(q):=\inf_{q_0\in M}h_{q_0,v(q_0)}(q,t),\\
T^{+}_t v(q):=\sup_{q_0\in M}h^{q_0,v(q_0)}(q,t).
\end{split}
\end{equation}
In addition, we set $T^{\pm}_0 v(q)=v(q)$, then for $t\geq0,\,\,T^{\pm}_t: v\mapsto T^{\pm}_t v$ maps $C(M,\R)$ to itself.
\end{definition}

The above definition allow us to deduce some properties of solutions semigroups from Proposition \ref{fundamental-prop} as corollaries. In particular, we have
\begin{proposition}\label{prop-sg}\cite[Proposition 4.3]{WWY2}
Two families of operator $\{T^{\pm}_t\}_{t\geqslant0}$ defined above satisfy
\begin{enumerate}
    \item[(1)]\textbf{(monotonicity)} For initial data $v,v'\in C(M,\R)$ with $v<v'$ (resp. $v\leqslant v'$) on $M$, then for all $q\in M$,
    \begin{equation}\label{mono-sg}
	T^{\pm}_t v(q)<T^{\pm}_t v'(q),\quad\text{resp.}\,\,(T^{\pm}_t v(q)\leqslant T^{\pm}_t v'(q)).
    \end{equation}

	\item[(2)]\textbf{(Semigroup property)} For any $t,\tau\geqslant0$,
    \begin{equation}\label{sg}
    T^{\pm}_{t+\tau}=T^{\pm}_{t}\circ T^{\pm}_{\tau},
    \end{equation}
    so that the families of operators $\{T^{\pm}_t\}_{t\geq0}$ form two \textbf{semigroups} acting on $C(M,\R)$.

	\item[(3)]\textbf{(Continuity 1)} For any $(q,t)\in M\times(0,+\infty)$, the functions
    \[
    (q,t)\mapsto T^{\pm}_t v(q),
    \]
    are locally Lipschitz continuous  and $\lim_{t\rightarrow0^+}T^{\pm}_t v(q)=v(q)$ for all $q\in M$.

    \item[(4)]\textbf{(Continuity 2)} For any $t\geqslant0$, the maps
    \[
    v\mapsto T^{\pm}_t v
    \]
    are continuous with respect to $\|\cdot\|_\infty$ defined on $C(M,\R)$.
\end{enumerate}

\end{proposition}

It turns out that the notion of subsolution (resp. strict subsolution) is equivalent to the $t$-monotonicity (-strict monotonicity) of the solution semigroups. The following proposition can be easily seen from the form of equation \eqref{HJe}.
\begin{proposition}\label{mono1}
Let $v\in C(M,\R)$ be a subsolution (resp. strict subsolution) to \eqref{HJs}, then
\begin{enumerate}[(1)]
  \item for any $q\in M$ and $t\geqslant0$ (resp. $t>0$), $v(q)\leqslant$\,\,(resp. $<$)\,\,$T^{-}_t v(q)$,
  \item for any $q\in M$ and $t\geqslant0$ (resp. $t>0$), $v(q)\geqslant$\,\,(resp. $>$)\,\,$T^{+}_t v(q)$.
\end{enumerate}
\end{proposition}

\subsection{Solution semigroups and their characteristics}
For a general initial data $v\in C(M,\R)$, we define $U:M\times[0,\infty)\rightarrow\R$ by
\begin{equation}\label{sol-sg}
U(q,t):=T^{-}_t v(q)=\inf_{q'\in M}h_{q',v(q')}(q,t).
\end{equation}
By Proposition \ref{fundamental-prop} (3), fixing $(q,t)\in M\times(0,+\infty)$, the map
\[
q'\mapsto h_{q',v(q')}(q,t)
\]
is continuous. Then there is a $q_0\in M$ such that $U(q,t)=h_{q_0,v(q_0)}(q,t)$. Due to the properties of backward action function, we have
\begin{lemma}\label{sol-action}
For any minimizer $\gamma:[0,t]\rightarrow M$ of $h_{q_0,v(q_0)}(q,t)$,
\[
U(\gamma(\tau),\tau)=h_{q_0,v(q_0)}(\gamma(\tau),\tau).
\]
\end{lemma}

\begin{proof}
By \eqref{sol-sg}, we only need to show that for any $\tau\in[0,t]$,
\[
U(\gamma(\tau),\tau)\geq h_{q_0,v(q_0)}(\gamma(\tau),\tau).
\]
We argue by contradiction. Assume there is $\tau_0\in(0,t)$ such that
\[
\underline{u}:=U(\gamma(\tau_0),\tau_0)<\bar{u}:=h_{q_0,v(q_0)}(\gamma(\tau_0),\tau_0),
\]
then to complete the proof, it is necessary to see that
\[
U(q,t)=T^-_{t-\tau_0}U(\cdot,\tau_0)(q)\leqslant h_{\gamma(\tau_0),\underline{u}}(q,t-\tau_0)<h_{\gamma(\tau_0),\bar{u}}(q,t-\tau_0)=h_{q_0,v(q_0)}(q,t)=U(q,t).
\]
Here, the first equality follows from property \eqref{sg} and the second equality is a consequence of Proposition \ref{fundamental-prop} (2) and the fact that $\gamma$ is a minimizer of $h_{q_0,v(q_0)}(q,t)$; the second inequality is deduced from Proposition \ref{fundamental-prop} (1).
\end{proof}

The following well-known theorem gives the name of the operator families defined in Definition \ref{semi-group}.
\begin{proposition}\label{sol}\cite[Proposition 4.4]{WWY2}
$U$ is the unique solution to \eqref{HJe}. In a similar fashion,
\[
U(q,t)=-T^{+}_t v(q)
\]
is the unique solution to
\begin{equation}\label{reverse ham}
\begin{cases}
\partial_t U+\breve{H}(q,\partial_q U,U)=0,\quad\, (q,t)\in M\times(0,+\infty),\\
\hspace{5.4em}U(\cdot,0)=v,\hspace{2.65em}q\in M,
\end{cases}
\end{equation}
where $\breve{H}(q,p,u)=H(q,-p,-u)$. Due to these facts, we call $\{T^{-}_t\}_{t\geqslant 0}$ the \textbf{backward solution semigroup} and $\{T^{+}_t\}_{t\geqslant 0}$ the \textbf{forward solution semigroup} to \eqref{HJe}.
\end{proposition}

We need the fact that the $q$-projection of the characteristic ensured by Theorem \ref{local-chm} is a minimizer in sense of Proposition \ref{Implicit variational}. Notice that by Theorem \ref{local-chm}, the map $\pi_q\circ\varphi^{\delta}_H:\Lambda_v\rightarrow M$ is a diffeomorphism, we use $(\pi_q\circ\varphi^{\delta}_H)^{-1}:M\rightarrow\Lambda_v$ to denote its inverse.

\begin{lemma}\label{local-mini}
For any $q_1\in M$ and $\sigma_0=(q_0,d_q v(q_0),v(q_0))=(\pi_q\circ\varphi^{\delta}_H)^{-1}(q_1)$, set
\[
\varphi^\tau_H\sigma_0=(q(\tau),p(\tau),u(\tau))\quad\text{for}\,\,\tau\in[0,\delta],
\]
then for all $\tau\in[0,\delta], u(\tau)=h_{q_0,v(q_0)}(q(\tau),\tau)$ and
\[
h_{q_0,v(q_0)}(q_1,\delta)=v(q_0)+\int_0^\delta L(q(\tau), \dot{q}(\tau),h_{q_0,v(q_0)}(q(\tau) ,\tau) )\ d\tau.
\]
\end{lemma}

\begin{proof}
By Theorem \ref{local-chm}, $U\in C^2(M\times[0,\delta],\R)$ and $(q(\tau),p(\tau),u(\tau))$ satisfy
\begin{equation}\label{eq:1}
p(\tau)=\partial_q U(q(\tau),\tau),\quad u(\tau)=U(q(\tau),\tau),\quad\tau\in[0,\delta].
\end{equation}
and the boundary conditions read as
\begin{equation}\label{eq:2}
q(0)=q_0,\quad q(\delta)=q_1,\quad u(0)=v(q_0).
\end{equation}
It follows from Proposition \ref{sol} that
\[
h_{q_0,v(q_0)}(q(\tau),\tau)\geq T^-_\tau v(q(\tau))=U(q(\tau),\tau)=u(\tau).
\]
Combining \eqref{eq:2} and Corollary \ref{Minimality} gives
\[
h_{q_0,v(q_0)}(q(\tau),\tau)=\inf\,\{u(\tau):(q(t),p(t),u(t))\in S^{q(\tau),\tau}_{q_0,v(q_0)}\}\leq u(\tau)
\]
and therefore for $\tau\in[0,\delta]$,
\[
u(\tau)=h_{q_0,v(q_0)}(q(\tau) ,\tau).
\]
Now we can compute as
\begin{align*}
&h_{q_0,v(q_0)}(q_1,\delta)\geq U(q_1,\delta)=U(q(\delta),\delta)\\
=\,&U(q(0),0)+\int^\delta_0 \partial_t U(q(\tau),\tau)+\langle\partial_q U(q(\tau),\tau),\dot{q}(\tau)\rangle\,d\tau\\
=\,&v(q_0)+\int^\delta_0 \partial_t U(q(\tau),\tau)+\langle p(\tau),\partial_p H(q(\tau),p(\tau),u(\tau))\rangle\,d\tau\\
=\,&v(q_0)+\int^\delta_0 \partial_t U(q(\tau),\tau)+H(q(\tau),p(\tau),u(\tau))+L(q(\tau),\dot{q}(\tau),u(\tau))\,d\tau\\
=\,&v(q_0)+\int^\delta_0 L(q(\tau),\dot{q}(\tau),u(\tau))\,d\tau\\
=\,&v(q_0)+\int^\delta_0 L(q(\tau),\dot{q}(\tau),h_{q_0,v(q_0)}(q(\tau) ,\tau))\,d\tau.
\end{align*}
Here, the third equality uses \eqref{eq:1} and the equations \eqref{ch} for characteristics, the fourth equality follows from the knowledge of Legendre-Fenchel inequality in convex analysis, i.e.,
\[
\dot{q}=\partial_p H(q,p,u)\Leftrightarrow\langle p,\dot{q}\rangle=H(q,p,u)+L(q,\dot{q},u),
\]
the fifth equality is due to \eqref{eq:1} and the fact that $U(q,t)$ is a solution to \eqref{HJe}, precisely
\begin{align*}
&\,\partial_t U(q(\tau),\tau)+H(q(\tau),p(\tau),u(\tau))\\
=&\,\partial_t U(q(\tau),\tau)+H(q(\tau),\partial_q U(q(\tau),\tau),U(q(\tau),\tau))=0.
\end{align*}
Combining the above inequality and \eqref{eq:Implicit1}, we complete the proof.
\end{proof}

The above lemma justifies the fact that the characteristics initiating from $\Lambda_v$ from which the local smooth solution is constructed by Theorem \ref{local-chm} are actually action minimizers of $T^-_t v$. This is true not only for local solutions, in fact we have
\begin{theorem}\label{global-chm}
Assume $v\in C^2(M,\R)$. For any $(q,t)\in M\times(0,+\infty)$, there is $\sigma_0\in\Lambda_v$ such that the characteristic segment $\varphi^\tau_H \sigma_0=(q(\tau),p(\tau),u(\tau)), \tau\in[0,t]$ satisfies
\[
U(q(\tau),\tau)=u(\tau).
\]
\end{theorem}

\begin{proof}
For $t\leq\delta$, Theorem \ref{global-chm} reduces to Theorem \ref{local-chm} and there is nothing to prove. For $t>\delta$, we use Definition \ref{semi-group} and \eqref{sg} to write
\[
U(q,t)=T^-_t v(q)=T^-_{t-\delta}\circ T^-_{\delta}v(q)=\inf_{q'\in M}h_{q',T^-_{\delta}v(q')}(q,t-\delta).
\]
Proposition \ref{fundamental-prop} (3) and \ref{prop-sg} (3) imply $h_{q',T^-_{\delta}v(q')}(q,t)$ is Lipschitz continuous in $q'$. Since $M$ is compact, the above infimum is attained at $q'=q_1\in M$. Set $u_1=T^-_{\delta}v(q_1)$, then according to Proposition \ref{Implicit variational}, there is a minimizer $\gamma_1:[0,t-\delta]\rightarrow M$ with $\gamma_1(0)=q_1$ and
\begin{equation}\label{min-1}
U(q,t)=h_{q_1,u_1}(q,t-\delta)=u_1+\int_0^{t-\delta} L(\gamma_1(\tau), \dot{\gamma_1}(\tau),h_{q_1,u_1}(\gamma_1(\tau) ,\tau))\ d\tau.
\end{equation}
For $\tau\in[\delta,t]$, we set $q_1(\tau)=\gamma_1(\tau-\delta)$ and
\[
u_1(\tau):=h_{q_1,u_1}(\gamma_1(\tau-\delta),\tau-\delta),\quad p_1(\tau):=\frac{\partial L}{\partial \dot q}(\gamma_1(\tau-\delta),\dot{\gamma}_1(\tau-\delta),u_1(\tau)),
\]
then Proposition \ref{Implicit variational} also implies that $(q_1(\tau),p_1(\tau),u_1(\tau))$ satisfies \eqref{ch} and
\begin{align*}
q_1(\delta)=q_1,\quad q_1(t)=q,\quad\lim_{\tau\rightarrow t^-}u_1(\tau)=u_1.
\end{align*}

\vspace{1em}
By Lemma \ref{local-mini}, there is a unique $\sigma_0=(q_0, d_q v(q_0), v(q_0))\in\Lambda_v$ such that
\[
\pi_q\circ\varphi^{\delta}_H(\sigma_0)=q_1,
\]
and if $\varphi^\tau_H\sigma_0=(q_0(\tau),p_0(\tau),u_0(\tau)),\tau\in[0,\delta]$, then
\begin{equation}\label{min-2}
h_{q_0,v(q_0)}(q_1,\delta)=v(q_0)+\int_0^\delta L(q_0(\tau), \dot{q}_0(\tau),h_{q_0,v(q_0)}(q_0(\tau) ,\tau) )\ d\tau
\end{equation}
and
\[
h_{q_0,v(q_0)}(q_1,\delta)=u_0(\delta)=U(q_1,\delta)=T^-_{\delta}v(q_1)=u_1.
\]

\vspace{1em}
\textbf{Claim:}\,\,for $\tau\in[\delta,t]$,
\begin{equation}\label{eq:3}
h_{q_0,v(q_0)}(q_1(\tau),\tau)=h_{q_1,u_1}(q_1(\tau),\tau-\delta).
\end{equation}

\vspace{1em}
\textit{Proof of the claim}:\,\,It follows from Proposition \ref{fundamental-prop} (2) that
\[
h_{q_0,v(q_0)}(q_1(\tau),\tau)=\inf_{q'\in M}h_{q',h_{q_0,v(q_0)}(q',\delta)}(q_1(\tau),\tau-\delta)\leq h_{q_1,u_1}(q_1(\tau),\tau-\delta).
\]
Assume for some $\tau_0\in(\delta,t)$,
\[
u_2:=h_{q_0,v(q_0)}(q_1(\tau_0),\tau_0)<h_{q_1,u_1}(q_1(\tau_0),\tau_0-\delta):=\bar{u}_2,
\]
then by Proposition \ref{fundamental-prop} (1) and the fact that $\gamma_1$ is a minimizer of $h_{q_1,u_1}(q,t-\delta)$,
\[
h_{q_0,v(q_0)}(q,t)\leq h_{q_1(\tau_0),u_2}(q,t-\tau_0)<h_{q_1(\tau_0),\bar{u}_2}(q,t-\tau_0)=h_{q_1,u_1}(q,t-\delta)=U(q,t).
\]
This contradicts to Proposition \ref{sol} since
\[
U(q,t)=T^-_t v(q)=\inf_{q'\in M}h_{q',v(q')}(q,t)\leq h_{q_0,v(q_0)}(q,t).
\]
\qed

\vspace{1em}
Define $q:[0,t]\rightarrow M$ by
\[
q(\tau):=
\begin{cases}
q_0(\tau),\quad \tau\in[0,\delta];\\
q_1(\tau),\quad \tau\in[\delta,t],
\end{cases}
\]
then combining \eqref{min-1}-\eqref{eq:3}, we obtain
\begin{align*}
&\,h_{q_0,v(q_0)}(q,t)=h_{q_1,u_1}(q,t-\delta)=u_1+\int_0^{t-\delta} L(\gamma_1(\tau), \dot{\gamma_1}(\tau),h_{q_1,u_1}(\gamma_1(\tau) ,\tau))\ d\tau\\
=&\,\,u_1+\int_{\delta}^{t} L(q_1(\tau), \dot{q_1}(\tau),h_{q_1,u_1}(q_1(\tau),\tau-\delta))\ d\tau\\
=&\,\,h_{q_0,v(q_0)}(q_1,\delta)+\int_{\delta}^{t} L(q_1(\tau), \dot{q_1}(\tau),h_{q_0,v(q_0)}(q_1(\tau),\tau))\ d\tau\\
=&\,\,v(q_0)+\int_0^\delta L(q_0(\tau), \dot{q}_0(\tau),h_{q_0,v(q_0)}(q_0(\tau) ,\tau) )\ d\tau+\int_{\delta}^{t} L(q_1(\tau), \dot{q_1}(\tau),h_{q_0,v(q_0)}(q_1(\tau),\tau))\ d\tau\\
=&\,\,v(q_0)+\int_0^t L(q(\tau), \dot{q}(\tau),h_{q_0,v(q_0)}(q(\tau),\tau))\ d\tau,\\
\end{align*}
This shows that $q:[0,t]\rightarrow M$ is a minimizer of $h_{q_0,v(q_0)}(q,t)$, thus by Proposition \ref{Implicit variational},
\[
u(\tau)=h_{q_0,v(q_0)}(\gamma(\tau),\tau),\quad p(\tau)=\frac{\partial L}{\partial \dot q}(q(\tau),\dot{q}(\tau),u(\tau))
\]
is a $C^1$ characteristic starting from $\sigma_0$. Invoking Lemma \ref{sol-action}, we find
\[
u(\tau)=
\begin{cases}
u_0(\tau)=U(q_0(\tau),\tau)=U(q(\tau),\tau),\quad\tau\in[0,\delta];\\
h_{q_1,u_1}(q_1(\tau),\tau-\delta)=U(q_1(\tau),\tau)=U(q(\tau),\tau),\quad\tau\in[\delta,t].
\end{cases}
\]
\end{proof}

\begin{remark}
The theorem shows that even for large $t$, the solution of \eqref{HJe} can also be traced by characteristics starting from the 1-graph of the initial data. The main difference from the case when $t$ is small, treated in Theorem \ref{local-chm}, is that the map $\pi_q\circ\varphi^t_H:\Lambda_v\rightarrow M$ is only a \textbf{surjection} rather than a diffeomorphism.
\end{remark}

\section{Connecting Legendrian graph and equilibria}
Let $u_-\in C^\infty(M,\R)$ be a classical solution to \eqref{HJs}, then $(H,\Lambda_{-})$ is a non-equilibrium thermodynamic system in the sense of Definition \ref{non-equilibrium}. This section is devoted to establishing abstract mechanisms for the existence of connecting orbits of an arbitrary Legendrian graph to the set of equilibria $\Lambda_{-}$. These mechanisms are based on the large time behavior of solution semigroups.

\subsection{Large time behavior of solution semigroups}
According to their definitions, $\{T^{\pm}_t\}_{t\geqslant 0}$ act on $C(M,\R)$, the space of continuous functions on $M$. We shall focus on the fixed points of such actions and introduce
\begin{definition}\label{fp}
A continuous function $u_{-}\,\,($ resp. $u_+)$ is called a fixed point of $\{T^{-}_t\}_{t\geqslant 0}\,\,($ resp. $\{T^{+}_t\}_{t\geqslant 0})$ if
\[
T^{-}_t u_-=u_-,\quad(\text{resp.}\,\,T^{+}_t u_+=u_+.)\quad\text{for any}\,\,t\geq0.
\]
We use $\mathcal{S}_-\,\,($resp. $\mathcal{S}_+)$ to denote the set of fixed points of $\{T^{-}_t\}_{t\geqslant 0}\,\,($resp. $\{T^{+}_t\}_{t\geqslant 0})$.
\end{definition}
\begin{remark}
  $\mathcal{S}_-$ is the analogy of weak KAM solutions \cite{Fathi_book} with $u$-independent  Hamiltonian.
\end{remark}

As an easy consequence of Proposition \ref{sol}, we have
\begin{proposition}\label{cov1}
$u_-\in\mathcal{S}_-$ if and only if $u_-$ is a solution to \eqref{HJs}. Similarly, $u_+\in\mathcal{S}_+$ if and only if $-u_+$ is a solution to the equation
\begin{equation}\label{reverse eq}
\breve{H}(q,d_q u(q),u(q))=0,\quad\text{for any }q\in M;
\end{equation}
\end{proposition}
If for some initial data $v\in C(M,\R)$, the uniform limit $u_-:=\lim_{t\rightarrow\infty}T^{-}_t v$ exists, then for any $s\geqslant0$, we deduce from Proposition \ref{prop-sg} (3) and (5) of $\{T^-_t\}_{t\geqslant0}$ that
\[
T^-_s u_{-}=T^-_s(\lim_{t\rightarrow\infty}T^{-}_t v)=\lim_{t\rightarrow\infty}T^-_s\circ T^{-}_t v=\lim_{t\rightarrow\infty}T^-_{s+t}v=u_-,
\]
so that $u_-\in\mathcal{S}_-$. Similar conclusion holds with $-$ replaced by $+$ in the above discussion. From PDE aspects, the existence of uniform limits $\lim_{t\rightarrow\infty}T^{\pm}_t v$ is usually studied under the subject of \textbf{large time behaviors}.

\subsection{Construction of connecting orbits I}
In this part, we wish to extract a mechanism for producing \textbf{semi-infinite} connecting orbits between the Legendrian graph $\Lambda_{0}$ and the states of equilibrium $\Lambda_{-}$ when $u_-$ is the uniform limit of the solution semigroup initiating from the smooth data $u_0$. In fact, we could obtain the following
\begin{theorem}\label{graph-graph1}
Assume $u_0,u_-\in C^\infty(M,\R)$ and $(H,\Lambda_{-})$ is a system. If the equality
\begin{equation}\label{convergence1}
\lim_{t\rightarrow+\infty}T^-_t u_0(q)=u_-(q)
\end{equation}
holds uniformly for all $q\in M$, then
\begin{enumerate}[(1)]
  \item $\Lambda_-\subset\overline{\cup_{t\geqslant0}\varphi^t_H(\Lambda_0)}$,
  \item there is $\sigma_0\in\Lambda_{0}$ such that $\omega(\sigma_0)$ is a nonempty subset of $\Lambda_-$.
\end{enumerate}
\end{theorem}

\begin{proof}
(1)\,\, For any $\sigma=(\bar{q},d_q u_-(\bar{q}),u_-(\bar{q}))\in\Lambda_-$, by Theorem \ref{global-chm}, we choose, for $n\geq1, \sigma_n=(q_n, p_n, u_n)\in\Lambda_{0}$ such that the corresponding characteristic segments
\[
\varphi^{t}_H\sigma_n=(q_n(t), p_n(t), u_n(t)),\quad t\in[0,n]
\]
satisfies the identity
\begin{equation}\label{eq:4}
u_n(t)=T^-_t u_0(q_n(t))=h_{q_n(t-\tau),u_n(t-\tau)}(q_{n}(t),\tau),\quad \lim_{n\rightarrow\infty}q_n(n)=\bar{q},
\end{equation}
for all $\tau\in[0,t]$, where we use \eqref{eq:3}. Combining the above equations with the assumption \eqref{convergence1},
\begin{equation}\label{u-convegence}
\lim_{n\rightarrow\infty}u_n(n)=\lim_{n\rightarrow\infty}T^-_n u_0(q_n(n))=u_-(\bar{q}).
\end{equation}
Due to the uniform Lipschitz property of $\{T_n u_0\}_{n\geq1}$, the characteristic segments $\{\varphi^{t}_H\sigma_n\}_{n\geqslant1}$ are uniformly bounded in $\Sigma$, thus the sequence $\{p_n(n)\}_{n\geqslant1}$ are relatively compact.

\vspace{1em}
\textbf{Claim}:\,\,$\lim_{n\rightarrow\infty}p_n(n)=d_q u_-(\bar{q})$.

\vspace{1em}
\textit{Proof of the claim}: We argue by contradiction to assume that for a subsequence $\{n_j\}\subset\mathbb{N}$ with $\lim_{j\rightarrow\infty}n_j=+\infty$ such that $\lim_{j\rightarrow\infty}p_{n_j}(n_j):=\bar{p}\neq d_q u_-(\bar{q})$. Set $\bar{\sigma}=(\bar{q},\bar{p},u_-(\bar{q}))$ and
\[
\varphi^{-1}_H{\bar\sigma}=(q_{-1},p_{-1},u_{-1}),\quad \varphi^{1}_H\sigma=(q_1,p_1,u_1).
\]
Due to the invariance of $\Lambda_-$ under $\varphi^t_H, u_1=u_-(q_1), p_1=d_q u_-(q_1)$. Since $\varphi^t_H\sigma_{n_j}$ are characteristics, we could apply the theorem of continuous dependence of solutions on initial data to obtain $\lim_{j\rightarrow\infty}q_{n_j}(n_j-1)=q_{-1}$ and arguing as \eqref{u-convegence} to obtain $u_{-1}=\lim_{j\rightarrow\infty}u_{n_j}(n_j-1)=u_-(q_{-1})$. Combining the above equality and \eqref{eq:4}, we deduce that
\[
u_-(\bar{q})=\lim_{j\rightarrow\infty}u_{n_j}(n_j)=\lim_{j\rightarrow\infty}h_{q_{n_j}(n_j-1),u_{n_j}(n_j-1)}(q_{n_j}(n_j),1)=h_{q_{-1},u_-(q_{-1})}(\bar{q},1),
\]
where the last equality follows from Proposition \ref{fundamental-prop}. Now we compute as
\[
u_-(q_1)=T^-_2 u_-(q_1)\leq h_{q_{-1},u_-(q_{-1})}(q_1,2)<h_{\bar{q},h_{q_{-1},u_-(q_{-1})}(\bar{q},1)}(q_1,1)=h_{\bar{q},u_{-}(\bar{q})}(q_1,1)=u_-(q_1),
\]
where the first equality uses the fact that $u_-\in\mathcal{S}_-$ and the second (strict) inequality uses the assumption $\bar{p}\neq d_q u_-(\bar{q})$ and the Markov property, i.e., Proposition \ref{fundamental-prop} (2), in fact the concatenate curve constructed from the minimizers of $h_{q_{-1},u_-(q_{-1})}(\bar{q},1)$ and $h_{\bar{q},h_{q_{-1},u_-(q_{-1})}(\bar{q},1)}(q_1,1)$ has a corner at $\bar{q}$, thus can not be the minimizer (must be $C^1$) of $h_{q_{-1},u_-(q_{-1})}(q_1,2)$. The situation is depicted below and it leads to a contradiction.

\begin{figure}[h]
  \begin{center}
  \includegraphics[width=8cm]{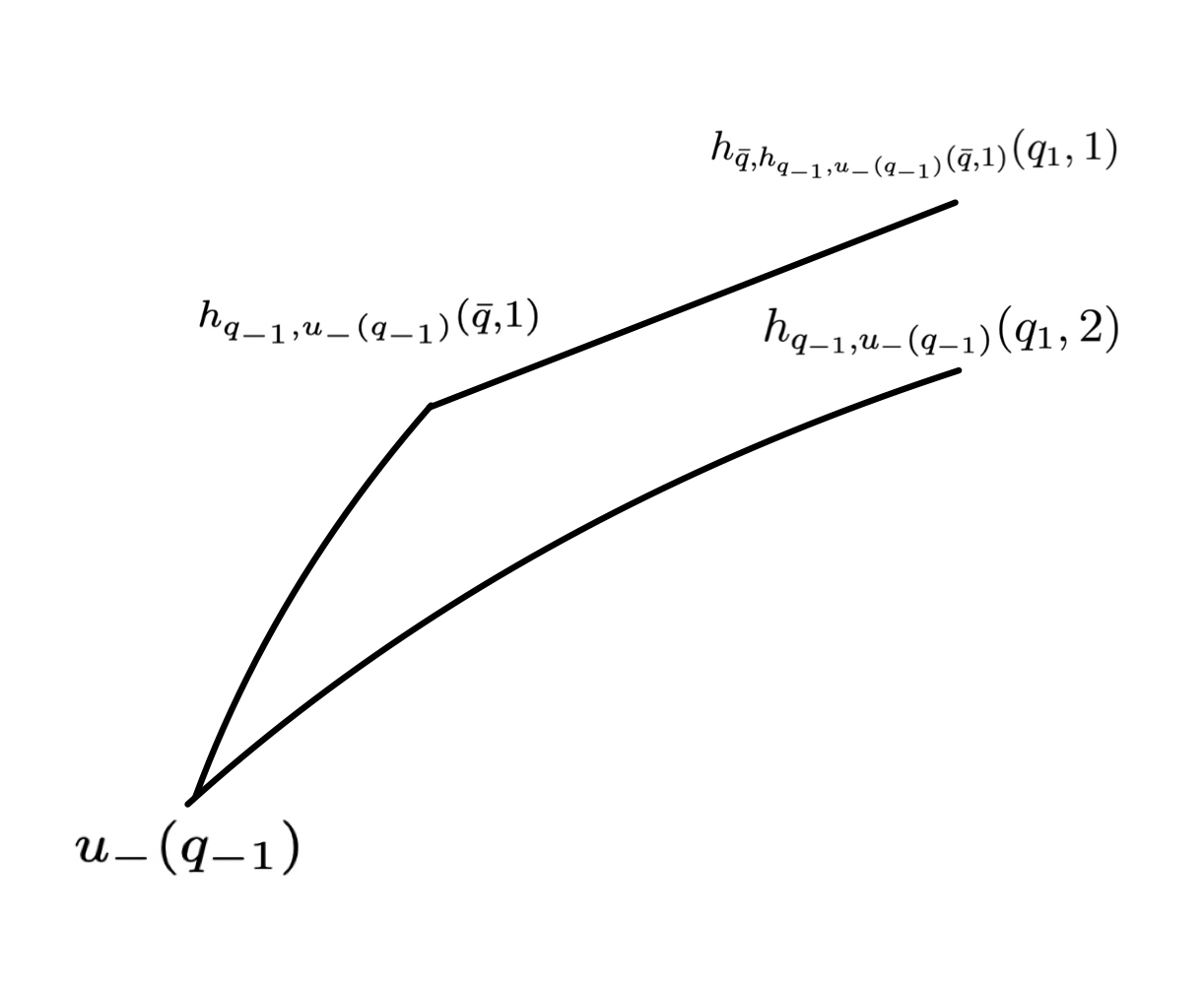}
  \end{center}
\end{figure}

(2)\,\,As in the proof of (1), we choose characteristic segments
\[
\sigma_n(\tau):=\varphi^{\tau}_H\sigma_n=(q_n(\tau), p_n(\tau), u_n(\tau)),\quad\tau\in[0,n]
\]
with $\sigma_n=(q_n, p_n, u_n)\in\Lambda_{0}$ and, up to a subsequence,
\begin{equation}\label{cali}
\lim_{n\rightarrow\infty}\sigma_n=\sigma_0=(q_0, d_q u_0(q_0), u_0(q_0))\in\Lambda_{0},\quad u_n(\tau)=T^-_\tau u_0(q_n(\tau)),
\end{equation}
here we add no assumption on the convergence of $q$-component. By continuous dependence of solutions of \eqref{ch} on the initial data, $\sigma_n(\tau)$ converges on compact intervals to a semi-infinite characteristic
\[
\sigma:[0,+\infty)\rightarrow\Sigma\quad\text{with}\quad \sigma(\tau)=(q(\tau),p(\tau),u(\tau))\quad\text{and}\quad \sigma(0)=\sigma_0.
\]
It follows from \eqref{cali} and the continuity of the function $T^-_\tau u_0, u(\tau)=T^-_\tau u_0(q(\tau))$ for all $\tau\in[0,+\infty)$. Due to the uniform Lipschitz property of $\{T_n u_0\}_{n\geq1}$, the characteristics $\sigma_n$ are uniformly bounded. This fact shows that $\sigma$ is bounded and $\omega(\sigma_0)$ is nonempty. Now we prove that $\omega(\sigma_0)\subset\Lambda_{-}$.

\vspace{1em}For any $\bar{\sigma}=(\bar{q},\bar{p},\bar{u})\in\omega(\sigma_0)$,\\

\textbf{Claim 1:}\,\,$\bar{u}=u_-(\bar{q})$. By definition, here is non-negative sequence $\{t_j\}_{j\geq1}$ with $\lim_{j\rightarrow\infty}t_j=+\infty$ and $(\bar{q},\bar{p},\bar{u})=\lim_{j\rightarrow\infty}(q(t_j),p(t_j),u(t_j))=\lim_{j\rightarrow\infty}(q(t_j),p(t_j),T^-_{t_j}u_0(q(t_j)))$. Thus for any $\epsilon>0$, there is $N_1\in\mathbb{N}$ such that for $j>N_1$,
\[
|\bar{u}-u(t_j)|<\epsilon,\quad |u_-(q(t_j))-u_-(\bar{q})|<\epsilon.
\]
On the other hand, \eqref{convergence1} implies that there is $N_2\in\mathbb{N}$ such that for $j>N_2$,
\[
|T^-_{t_j}u_0(q)-u_-(q)|<\epsilon,\quad\text{for all}\quad q\in M.
\]
Thus for $j>\max\{N_1,N_2\}$,
\begin{align*}
|\bar{u}-u_-(\bar{q})|&\,\leq|\bar{u}-u(t_j)|+|u(t_j)-u_-(q(t_j))|+|u_-(q(t_j))-u_-(\bar{q})|\\
&\,\leq|\bar{u}-u(t_j)|+|T^-_{t_j}u_0(q(t_j))-u_-(q(t_j))|+|u_-(q(t_j))-u_-(\bar{q})|\\
&\,<3\epsilon,
\end{align*}
this shows the first claim.

\vspace{1em}
\textbf{Claim 2:}\,\,$\bar{p}=d_q u_-(\bar{q})$. Notice that if $\bar{\sigma}\in\omega(\sigma_0)$, then for $\tau\in\R$,
\[
\bar{\sigma}(\tau)=\varphi^{\tau}_H\bar{x}=(\bar{q}(\tau),\bar{p}(\tau),\bar{u}(\tau))\in\omega(\sigma_0).
\]
By the first claim, we have $\bar{u}(\tau)=u_-(\bar{q}(\tau))$ for every $\tau\in\R$. For $\tau\in\R$, we set
\[
(\tilde{q}(\tau),\tilde{p}(\tau),\tilde{u}(\tau))=\varphi^{\tau}_H(\bar{q},d_q u_-(\bar{q}),\bar{u}).
\]
It follows that $\bar{q}(0)=\tilde{q}(0)=\bar{q}, \bar{u}(0)=\tilde{u}(0)=u_-(\bar{q})$. Since $u_-$ is a $C^2$ solution to \eqref{HJs},
\[
(\tilde{q}(\tau),\tilde{p}(\tau),\tilde{u}(\tau))\in\Lambda_-.
\]

\vspace{1em}
We argue as in (1) to assume that $\bar{p}\neq d_q u_-(\bar{q})$. Since $u_-\in\mathcal{S}_-$, then
\[
\bar{u}=u_-(\bar{q})=T^-_1 u_-(\bar{q})\leq h_{\bar{q}(-1),u_-(\bar{q}(-1))}(\bar{q},1)\leq\bar{u},
\]
where the last inequality follows from Corollary \ref{Minimality}, and
\begin{align*}
u_-(\tilde{q}(1))&\,=T^-_2 u_-(\tilde{q}(1))\leq h_{\bar{q}(-1),u_-(\bar{q}(-1))}(\tilde{q}(1),2)\\
&\,<h_{\bar{q},h_{\bar{q}(-1),u_-(\bar{q}(-1))}(\bar{q},1)}(\tilde{q}(1),1)=h_{\bar{q},\bar{u}}(\tilde{q}(1),1)\\
&\,=h_{\tilde{q}(0),u_-(\tilde{q}(0))}(\tilde{q}(1),1)\leq\tilde{u}(1)=u_-(\tilde{q}(1)),\\
\end{align*}
where the second inequality uses Markov property, i.e. Proposition \ref{fundamental-prop} (ii), of action function and the last inequality is again a consequence of Corollary \ref{Minimality}. This leads to the contradiction we desire.
\end{proof}

\begin{remark}
The above mechanism has the advantages that it requires no more information about the local dynamics of $\Lambda_-$. However, the condition \eqref{convergence1} asserts the convergence of the orbit $\{T^-_t u_0\}$ to a fixed solution $u_-$, which may hold only for a small part of initial data if there is no additional assumption on $H$.
\end{remark}

\subsection{Construction of connecting orbits II}
In this second part, we try to weaken the condition \eqref{convergence1} under the local stability assumption on the set of equilibria. As is mentioned in the introduction, such assumption is widely adopted by physicists working in non-equilibrium thermodynamics.
\begin{theorem}\label{graph-graph2}
Assume $u_0,u_-\in C^\infty(M,\R), (H,\Lambda_-)$ is a system satisfying \textbf{(H1)} and the limit
\begin{equation}\label{covergence2}
u_\ast(q)=\liminf_{t\rightarrow\infty}\,\,\,T^-_t u_0(q)
\end{equation}
exists uniformly in $q\in M$ with
\begin{enumerate}[(1)]
  \item $u_\ast\geq u_-$ on $M$,
  \item the set $\{q\in M\,:\, u_\ast(q)=u_-(q)\}$ is nonempty,
\end{enumerate}
then there exists $\sigma_0\in\Lambda_{0}$ such that $\omega(\sigma_0)$ is a nonempty subset of $\Lambda_-$.
\end{theorem}

\begin{proof}
Since $\Lambda_-$ is a local attractor, it remains to show the existence of $\sigma_0\in\Lambda_0$ such that for some $T>0$,
\[
\varphi^{T}_H \sigma_0\in\mathcal{O}_-.
\]
For each $n\geq1$, we choose $q_n\in M, \{t_n\}_{n\in \mathbb{N}} $ such that $\lim_{n\rightarrow\infty}T^-_{t_n} u_0(q)=\liminf_{t\rightarrow\infty}T^-_t u_0(q)$ and
\begin{equation}\label{max-mini}
T^-_{t_n} u_0(q_n)-u_-(q_n)=\min_{q\in M}\{T^-_{t_n} u_0(q)-u_-(q)\}.
\end{equation}
Since $T^-_{t_n} u_0-u_-$ is a semi-concave function on $M$, it is differentiable at the minimal point $q_n$ and
\begin{equation}\label{derivative}
d_q T^-_{t_n} u_0(q_n)=d_q u_-(q_n).
\end{equation}
We set $\sigma_n:=(q_n, d_q T^-_{t_n} u_0(q_n), T^-_{t_n} u_0(q_n))$, then\\

\textbf{Claim}: Any limit point of the sequence $\{\sigma_n\}_{n\geq1}$ belongs to $\Lambda_-$.\\

\vspace{0.3em}
Up to a subsequence, we may assume $\lim_{n\rightarrow\infty}q_n=\bar{q}$ and $\lim_{n\rightarrow\infty}T^-_{t_n} u_0(q_n)$ exists. Applying \eqref{derivative} and the fact that $u_-\in C^\infty(M,\R)$,
\[
\lim_{n\rightarrow\infty}d_q T^-_{t_n} u_0(q_n)=d_q u_-(\bar{q}).
\]
By the definition of $u_\ast, q_n$ as well as the equi-Lipschitz continuity of $\{T^-_{t_n} u_0\}_{n\geq1}$, for any $q\in M$,
\begin{align*}
&\,\,\,u_\ast(q)-u_-(q)\\
=&\,\lim_{n\rightarrow\infty}[T^-_{t_n} u_0(q)-u_-(q)]\geq\lim_{n\rightarrow\infty}[T^-_{t_n} u_0(q_n)-u_-(q_n)]\\
=&\,\lim_{n\rightarrow\infty}[T^-_{t_n} u_0(\bar{q})-u_-(\bar{q})]= u_\ast(\bar{q})-u_-(\bar{q}).
\end{align*}
By the fact that $u_\ast\geq u_-$ and the set $\{q\in M:u_\ast(q)=u_-(q)\}$ is nonempty, we have
\[
u_\ast(\bar{q})=u_-(\bar{q}),\quad\lim_{n\rightarrow\infty}T^-_{t_n} u_0(q_n)=u_-(\bar{q}).
\]
This verifies the claim. Thus for all $n$ sufficiently large, $\sigma_n\in\mathcal{O}_-$. To conclude this case, we fix such an $\sigma_n$ and use Theorem \ref{global-chm} to find $\sigma_n(0):=\varphi^{-t_n}_H\sigma_n\in\Lambda_0$ such that the corresponding characteristic segments $\varphi^{\tau}_H\sigma_n(0), \tau\in[0,t_n]$ connects $\Lambda_0$ with $\sigma_n$.
\end{proof}

\section{Homotopic criteria for the convergence of solution semigroup}
Due to the abstract mechanisms established in the last section, the problem of constructing semi-infinite orbits connecting an arbitrary Legendrian graph $\Lambda_0$ to the equilibria submanifold $\Lambda_-$ is reduced to the study of large time behavior of the generating data $u_0$. The main goal of this section is to provide some criteria to ensure the required behavior of $u_0$. We divide this section in two parts, according to whether the local stability is assumed for the set of equilibria.

\subsection{Homotopy method I}
Due to the definition of sub and super-deformation defined in the introduction, one could use the $t$-monotonicity of solution semigroups, indicated in Proposition \ref{mono1}, to show the following
\begin{theorem}\label{homotopy1}
Assume $u_0\in C(M,\R)$ and $(H,\Lambda_-)$ is a system. If there exists
\begin{enumerate}[(a)]
  \item   a sub-deformation $V:M\times[0,1]\rightarrow\R$ of $u_-$ such that
        \[
        V(q,0)\leq u_0(q)\leq u_-(q)\quad\text{for any }q\in M,
        \]

  \item  a super-deformation $V:M\times[0,1]\rightarrow\R$ of $u_-$ such that
        \[
        V(q,0)\geq u_0(q)\geq u_-(q)\quad\text{for any }q\in M,
        \]

  \item  a sub-deformation $\underline V:M\times[0,1]\rightarrow\R$ and  a super-deformation $\overline V:M\times[0,1]\rightarrow\R$ of $u_-$  such that
   $$
  \underline  V(\cdot ,0 ) \leqslant \min\{ u_0, u_-\}, \quad \max\{ u_0, u_-\}   \leqslant \overline V(\cdot ,0 ) .
   $$
\end{enumerate}
then \,\,$\lim_{t\rightarrow+\infty}T^-_t u_0=u_-$.
\end{theorem}

\begin{proof}
We shall show the conclusion holds under the assumption (a) and the proof under (b) is completely similar. Due to Proposition \ref{prop-sg} (1), we obtain that
\begin{equation}\label{sol-mono}
T^-_t V(q,0)\leq T^-_t u_0(q)\leq T^-_t u_-(q)=u_-(q),\quad\text{for any }\,\,q\in M.
\end{equation}
Thus to prove the conclusion, it is enough to prove
\begin{equation}\label{convergence3}
\lim_{t\rightarrow+\infty}T^-_t V(\cdot,0)=u_-.
\end{equation}

Since $V(\cdot,0)$ is a subsolution to \eqref{HJs}, Proposition \ref{mono1} guarantees that for $0\leq t\leq\tau$,
\[
T^-_t V(\cdot,0)\leq T^-_\tau V(\cdot,0).
\]
Combining \eqref{sol-mono} and the equi-Lipschitz property of $\{T^-_t V(\cdot,0)\}_{t\geq1}$,
\begin{itemize}
  \item the uniform limit $v_-=\lim_{t\rightarrow+\infty}T^-_t V(\cdot,0)$ exists and $v_-\in\mathcal{S}_-$ with $V(\cdot,0)\leq v_-\leq u_-$ on $M$.
\end{itemize}
Assume $v_-\neq u_-$, then there are $s_0\in[0,1), q_0\in M$ such that
\[
V(q,s_0)\leq v_-(q)\quad\text{for any}\,\,q\in M,\quad\text{and}\quad V(q_0,s_0)=v_-(q_0)<u_-(q_0).
\]
Then by Definition \ref{sub-super-intro}, as a subsolution to \eqref{HJs}, $V(\cdot,s_0)$ is strict and applying Proposition \ref{mono1} again,
\[
T^-_t v_-(q_0)\geq T^-_t V(q_0,s_0)>V(q_0,s_0)=v_-(q_0)
\]
for all $t>0$, which contradicts the fact that $v_-\in\mathcal{S}_-$.

\vspace{1em}

  (c): Applying (a),(b) above to the assumption
$$
	\underline V(\cdot ,0 )  \leqslant \underline{u_0}:=\min\{u_0, u_{-}\}\leqslant u_- \leqslant\overline{u_0}:= \max\{u_0, u_{-} \} \leqslant \overline V(\cdot ,0 ),
$$
we arrive at
$$
\lim_{t\rightarrow+\infty}T^-_t  \underline{u_0}=u_-, \quad \lim_{t\rightarrow+\infty}T^-_t \overline{u_0}=u_-.
$$
Due to $\underline{ u_0} \leqslant u_0 \leqslant  \overline{ u_0}  $, one follows that $\lim_{t\rightarrow+\infty}T^-_t  u_0=u_- $.
\end{proof}

\subsection{Homotopy method II}
If the priori stability of the manifold of equilibria is assumed, then we could replace the existence of sub \,\,(resp. super)-deformation in Theorem \ref{homotopy1} by a more flexible condition. So we propose a second
\begin{theorem}\label{homotopy2}
Assume $u_0\in C(M,\R)$ and $(H,\Lambda_-)$ is a system.
 Assume one of the following conditions
\begin{enumerate}[(a)]
  \item[$(a')$] there is a continuous function $V:M\times[0,1]\rightarrow\R$ with $V(\cdot ,0 )$ is a $C^2$ subsolution to \eqref{HJs} and
        \[
        	 H(q,d_q u_-(q),V(q,s))<0,  \quad \forall q\in M \quad   \text{and}\quad V(\cdot,0) \leq u_0\leq u_-=V(\cdot,1),
        \]

  \item[$(b')$] there is a continuous function $V:M\times[0,1]\rightarrow\R$ with $V(\cdot ,0 )$ is a $C^2$ supersolution to \eqref{HJs} and
      \[
        	 H(q,d_q u_-(q),V(q,s)) >0,  \quad \forall q\in M \quad   \text{and}\quad V(\cdot,0) \geq u_0\geq u_-=V(\cdot,1),
        \]
  \item [$(c')$] there are continuous functions $\underline{V}, \overline{V}:M\times[0,1]\rightarrow\R$ such that $\underline{V}(\cdot ,0 ),\overline{V}(\cdot ,0 )$ are a $C^2$ subsolution and a $C^2$ supersolution to \eqref{HJs} respectively and for all $q\in M$,
        \begin{align*}
        & H(q,d_q u_-(q),\underline{V}(q,s))<0,\quad\quad H(q,d_q u_-(q),\overline{V}(q,s))>0,\\
        &\underline{V}(q,0)\leqslant\min\{u_0(q), u_-(q)\}, \quad\max\{u_0(q), u_-(q)\}  \leqslant\overline{V}(q,0).
        \end{align*}
\end{enumerate}
is satisfied, then the uniform limit $u_\ast(q)=\liminf_{t\rightarrow\infty}T^-_t u_0(q)$ exists and the set
\[
\{q\in M\,:\,u_\ast(q)=u_-(q)\}\quad\text{is nonempty.}
\]

\end{theorem}

\begin{proof}
$(a')$: Since $V(\cdot,0) \leq u_0\leq u_-$ on $M$, due to Proposition \ref{prop-sg} (1), we have
\begin{equation}\label{sol-mono1}
T^-_t V(q,0) \leq T^-_t u_0(q)\leq T^-_t u_-(q)=u_-(q),\quad\text{for any }\,\,q\in M.
\end{equation}
Using the assumption that   $ V(\cdot,0)$ is a subsolution to \eqref{HJs},  Proposition \ref{mono1} implies that for $0\leq t\leq\tau$,
\begin{equation}\label{mono-t}
T^-_t  V(\cdot,0) \leq T^-_\tau  V(\cdot,0).
\end{equation}
Combining \eqref{sol-mono1} and the equi-Lipschitz property of $\{T^-_t  V(\cdot,0)\}_{t\geq1}$, the uniform limit
$$
v_\ast:=\lim_{t\rightarrow+\infty}T^-_t  V(\cdot,0)
$$
exists with $v_\ast\in\mathcal{S}_-$ and
\[
V(\cdot,0)\leq v_\ast\leq u_-\quad\text{on}\,\,M.
\]
We now show that $v_\ast=u_-$ by contradiction. Assume there is $q_0\in M$ such that
\begin{equation}\label{max}
u_-(q_0)-v_\ast(q_0)=\max_{q\in M}[u_-(q_0)-v_\ast(q_0)]>0.
\end{equation}
By the continuity of the function $V$, there are $s_0\in[0,1),\delta>0$ such that $V(q_0,s_0)=v_\ast(q_0)$ and
\begin{align*}
&\,\sup_{\dot{q}\in T_{q_0}M}[\langle d_q u_-(q_0),\dot{q}\rangle-L(q_0,\dot{q},v_\ast(q_0))]\\
=&\,H(q_0,d_q u_-(q_0),v_\ast(q_0))=H(q_0,d_q u_-(q_0),V(q_0,s_0))=-\delta<0.
\end{align*}
By Proposition \ref{Implicit variational} and Definition \ref{semi-group}, there is a $\gamma\in C^1([-1,0],M)$ with $\gamma(0)=q_0$ such that
\[
T^-_t v_\ast(q_0)=h_{\gamma(-t),v_\ast(\gamma(-t))}(q_0,t)=v_\ast(\gamma(-t))+\int_{-t}^0 L(\gamma(\tau), \dot{\gamma}(\tau),h_{\gamma(-t),v_\ast(\gamma(-t))}(\gamma(\tau), \tau) )\ d\tau.
\]
holds for all $t\in[0,1]$. Due to the fact that $v_\ast\in\mathcal{S}_-$,\,\,for $t>0$ sufficiently small,
\begin{align*}
&\,u_-(q_0)-v_\ast(q_0)=u_-(q_0)-T^-_t v_\ast(q_0)\\
=&\,[u_-(\gamma(-t))-v_\ast(\gamma(-t))]+[u_-(\gamma(0))-u_-(\gamma(-t))]-\int_{-t}^0 L(\gamma(\tau), \dot{\gamma}(\tau),h_{\gamma(t),v_\ast(\gamma(t))}(\gamma(\tau) ,\tau)\ d\tau\\
=&\,[u_-(\gamma(-t))-v_\ast(\gamma(-t))]+t\,[\langle d_q u_-(q_0),\dot{\gamma}(0)\rangle-L(q_0,\dot{\gamma}(0),v_\ast(q_0))]+o(t)\\
\leq&\,[u_-(\gamma(-t))-v_\ast(\gamma(-t))]-\delta t+o(t)\\
<&\,u_-(\gamma(-t))-v_\ast(\gamma(-t))\leq u_-(q_0)-v_\ast(q_0),
\end{align*}
where the last inequality uses \eqref{max}. This leads to the contradiction.

\vspace{1em}
To complete the proof, we use \eqref{sol-mono1} to obtain
$$
u_-=v_\ast\leqslant u_\ast\leqslant\limsup_{t\rightarrow\infty}T^-_t u_0\leqslant  u_-,
$$
and it follows that $u_\ast(q)=v_\ast(q)=u_-(q)$ for any $q\in M$.

\vspace{1em}
$(b')$:\,\,By reversing the signs in the inequalities \eqref{sol-mono1}-\eqref{mono-t} and arguing with the same reasoning as in the beginning of the above proof, the uniform limit $v_\ast:=\lim_{t\rightarrow+\infty}T^-_t  V(\cdot,0 ) $ exists with $v_\ast\in\mathcal{S}_-$ and
\[
 V(\cdot,0 ) \geq v_\ast\geq u_-\quad\text{on}\,\,M.
\]
It follows from the fact that $V(0,\cdot)\geqslant u_0$ and Proposition \ref{prop-sg} (1) that
\[
v_\ast\geqslant\limsup_{t\rightarrow\infty}T^-_t u_0\geqslant\liminf_{t\rightarrow\infty}T^-_t u_0=u_\ast\geqslant u_-
\]
To complete the proof, it is enough show that the set $\{q\in M\,:\,v_\ast(q)=u_-(q)\}$ is nonempty, we argue by contradiction. Assume $v_\ast>u_-$ on $M$ and there is $q_0\in M$ with
\begin{equation}\label{min-3}
v_\ast(q_0)-u_-(q_0)=\min_{q\in M}\,\,[v_\ast(q)-u_-(q)]>0.
\end{equation}
By the continuity of the function $V$, there is $s_0\in[0,1)$ such that $V(q_0,s_0)=v_\ast(q_0)$ and
\[
\sup_{\dot{q}\in T_{q_0}M}[\langle d_q u_-(q_0),\dot{q}\rangle-L(q_0,\dot{q},v_\ast(q_0))]=H(q_0,d_q u_-(q_0),v_\ast(q_0))=H(q_0,d_q u_-(q_0),V(q_0,s_0))>0.
\]
Thus there is $\dot{q}_0\in T_{q_0}M$ such that
\[
\delta:=\langle d_q u_-(q_0),\dot{q}_0\rangle-L(q_0,\dot{q}_0,v_\ast(q_0))>0.
\]
Setting $(q(t),p(t),u(t))=\varphi^t_H(q_0,p_0,v_\ast(q_0))$, where $p_0=\frac{\partial L}{\partial\dot{q}}(q_0,\dot{q}_0,v_\ast(q_0))$, then for $t>0$ sufficiently small, we have
\begin{align*}
&\,u(t)-u_-(q(t))=[u(0)+t\dot{u}(0)+o(t)]-[u_-(q(0))+t\langle d_q u_-(q(0)),\dot{q}(0)\rangle+o(t)]\\
=&\,[v_\ast(q_0)-u_-(q_0)]-[\langle d_q u_-(q_0),\dot{q}_0\rangle-L(q_0,\dot{q}_0,v_\ast(q_0))]t+o(t)\\
=&\,[v_\ast(q_0)-u_-(q_0)]-\delta t+o(t),
\end{align*}
where for the second equation, the charactersitic system \eqref{ch} as well as the definition of $p_0$ are used. Since $v_\ast\in\mathcal{S}_-$, it follows from Definition \ref{semi-group} and Corollary \ref{Minimality} that
\begin{align*}
&\,v_\ast(q(t))-u_-(q(t))=T^-_t v_\ast(q(t))-u_-(q(t))\\
\leq&\, h_{q_0,v_\ast(q_0)}(q(t),t)-u_-(q(t))\leq u(t)-u_-(q(t))\\
=&\,[v_\ast(q_0)-u_-(q_0)]-\delta t+o(t)\\
<&\,v_\ast(q_0)-u_-(q_0),
\end{align*}
which contradicts \eqref{min-3}.

\vspace{1em}
$(c')$: The proof is completely similar to that of Theorem \ref{homotopy1} (c).
\end{proof}

\section{Sample examples and applications of the main theorems}
This section includes direct consequences of the connecting mechanisms as well as interpretations of them on certain classes of contact systems. As concluding remarks, we shall discuss the relationship of our results with those obtained in the paper \cite{EP}.

\subsection{Monotone Hamiltonian}
Let us begin with the class of contact Hamiltonian $H$ that are strictly increasing in $u$, i.e.,
\begin{itemize}
  \item [\textbf{(M)}] $\frac{\partial H}{\partial u}(q,p,u)>0$ for every $(q,p,u)\in\Sigma$.
\end{itemize}
These Hamiltonian could be seen as a generalization of discounted Hamiltonian and the corresponding systems model the motions of particles in mechanical systems with friction. For a contact Hamiltonian system defined by monotone Hamiltonian, we could obtain the following conclusion which relates to the fact that $\Lambda_-$ is part of the maximal attractor for $\varphi^t_H$ on $\Sigma$.
\begin{corollary}\label{connect-mono}
Assume $H$ satisfies \textbf{(H2)} and \textbf{(M)} and there is a $C^\infty$ function $u_-:M\rightarrow\R$ such that $(H,\Lambda_-)$ constitutes a system. Then for every $C^\infty$ function $u_0$, there is $\sigma_0\in\Lambda_0$ such that $\omega(\sigma_0)\subset\Lambda_-$.
\end{corollary}

\begin{proof}
Due to compactness of $M$, there is $c>0$ such that for all $q\in M$,
\[
u_-(q)-c\leq u_0(q)\leq u_-(q)+c.
\]
Notice that $V_-(q,t)($ resp.$V_+(q,t))$ $:M\times[0,1]\rightarrow\R$ defined by
\[
V_-(q,t)=u_-(q)-(1-t)c,\quad(\text{resp.}\,\,V_+(q,t)=u_-(q)+(1-t)c)
\]
are a sub (resp. super)-deformation of $u_-$ respectively with $V_\pm(q,0)=u_-(q)\pm c$. Then by the monotonicity of solution semigroup, we have for $t\geq0$,
\[
T^-_t V_{-}(q,0)\leq T^-_t u_0(q)\leq T^-_t V_{+}(q,0)
\]
By Theorem \ref{homotopy1}, we have
\[
\lim_{t\rightarrow\infty}T^-_t V_{\pm}(q,0)=u_-(q)\quad\text{and so}\quad\lim_{t\rightarrow\infty}T^-_t u_0(q)=u_-(q).
\]
Now we apply Theorem \ref{graph-graph1} to complete the proof.
\end{proof}

\subsection{Contact M\"{o}bius model}
In this part, we apply our connecting mechanism to give an analysis of an interesting model raised in \cite{EP}.
\begin{example}\cite[Example 2.12]{EP}\label{contact Mobius}
Let $M=\mathbb{S}^1$ and $\Sigma=J^1 M$ the corresponding phase space with the canonical contact structure defined in the introduction. The Hamiltonian $F_{\mathcal{M}}:\Sigma\rightarrow\R$ defined by
\[
F_{\mathcal{M}}(q,p,u)=p^2+u^2-1
\]
induced an integrable contact Hamiltonian flow on $\Sigma$. Precisely, since the Hamiltonian is independent of $q$, the contact Hamiltonian vector field \eqref{ch} can be projected to the $(p,u)$-plane as
\[
\begin{cases}
\dot{p}=-2pu,\\
\dot{u}=p^2-u^2+1.
\end{cases}
\]
If one introduce the complex coordinate $z=u+\sqrt{-1}\,p$ on the $(p,u)$-plane with the real cylinder $l=\{p=0\}$, the flow defined by \eqref{cm-model} is described as the one-parameter subgroup of the M\"{o}bius transformations $PSL(2,\R)$ admitting an unstable fixed point at $w=-1$ and a stable point at $w=1$. In the complex coordinates, the solutions reads as
\begin{equation}\label{cm-model}
w(t)=\frac{w(0)\cosh t+\sinh t}{w(0)\sinh t+\cosh t},\quad w(0)=u(0)+\sqrt{-1}\,p(0).
\end{equation}
From the above formula, one could see that the phase flow of \eqref{cm-model} is incomplete. For $\epsilon>0$, we choose a cut-off function $a:[0,+\infty)\rightarrow[-1,+\infty)$ with $a'(s)>0$ for all $s$ and
\[
\lim_{s\rightarrow\infty}a(s)=a_{\infty}>1,\quad a(s)=s-1,\quad\text{for}\,\,s\in[0,1+\epsilon],
\]
to construct a new Hamiltonian $H_{\mathcal{M}}(q,p,u)=a(p^2+u^2)$ to make the flow complete with the dynamics in the disk $\{p^2+u^2<1+\epsilon\}$ unchanged. The authors focus on the following fact since it contains some ingredients for the mechanism of their constructions \cite[Theorem 2.9]{EP}.

\begin{itemize}
  \item [$(\clubsuit)$] Along the real cylinder $l=\{p=0\}$, for $c\in\R$, if we define Legendrian graphs
  \[
  \Lambda_-=\{(q,0,1)\,:\,q\in\mathbb{S}^1\},\quad \Lambda_c:=\{(q,0,c)\,:\,q\in\mathbb{S}^1\},
  \]
  then $\Lambda_-$ is a \textbf{local attractor} for $\varphi^t_H$ and $\Lambda_c$ admits trajectories of the contact Hamiltonian flow starting on $K_c$ and converge asymptotically to $\Lambda_-$ \textbf{for $c>1$ but not for $c<-1$}.
\end{itemize}
\end{example}
The author have shown that the Legendrian submanifolds $(\Lambda_c,\Lambda_-)$ is interlinked for $c>1$ \textbf{but not for $c<-1$}. This fact explains why the connecting mechanism \cite[Theorem 2.9]{EP} works only for $c>1$.

\vspace{1em}
To give an interpretation of the fact $(\clubsuit)$ from our viewpoint, we shall focus on the dynamics in a neighborhood of the unit disk on $(p,u)$-plane. In this region, $H_{\mathcal{M}}=F_{\mathcal{M}}=u^2+p^2-1$ satisfying \textbf{(H2)}. According to description of $(\clubsuit)$, we divide the analysis into two cases: for an initial data $u_0\in C^\infty(\mathbb{S}^1,\R)$ with
\begin{enumerate}
  \item $u_0>-1$ (not necessarily constant) and $\Lambda_0\subset\{p^2+u^2<1+\epsilon\}$, then
        \begin{equation}\label{eq:order2}
        -1<\underline{U}:=\min_{q\in M}\{u_0(q), u_{-}(q)\}\leqslant u_0 \leqslant\overline{U}:= \max_{q}\{u_0(q), u_{-}(q)\} \leqslant\sqrt{1+\epsilon}.
        \end{equation}
        Then one easily constructs
        \begin{itemize}
        \item a sub-deformation
        $
        \underline{V}:M\times[0,1]\rightarrow\R,\quad\underline{V}(q,\lambda)=(1-\lambda)\underline{U}+\lambda,
        $
        \item a super-deformation
        $
        \overline{V}:M\times[0,1]\rightarrow\R,\quad\overline{V}(q,\lambda)=(1-\lambda)\overline{U}+\lambda.
        $
        \end{itemize}
        Now we apply Theorem \ref{main1} (c) to get
        \begin{enumerate}[(1)]
        \item $\Lambda_-\subset\overline{\cup_{t\geqslant0}\varphi^t_H(\Lambda_0)}$,
        \item and there is $\sigma_0\in\Lambda_0$ such that $\omega(\sigma_0)\subset\Lambda_-$.
        \end{enumerate}

  \item $u_0\equiv c<-1$ and $\Lambda_0\subset\{p^2+u^2<1+\epsilon\}$, then $u_0<u_-$ and any $C^\infty$ function $V:M\times[0,1]\rightarrow\R$ with $V(\cdot,0)\leqslant u_0$ satisfies
        \[
        H_{\mathcal{M}}|_{\Lambda_{V(\cdot,0)}}>0,\quad H_{\mathcal{M}}(\cdot,0,V(\cdot,0))>0.
        \]
        From these facts, it is easily deduced that the deformations listed in the conditions $(a)-(c),(a')-(c')$ of Theorem \ref{main1}-\ref{main2} \textbf{do not exist}, thus showing the necessity of these conditions.
\end{enumerate}

By studying the phase portrait of the system defined by $H_{\mathcal{M}}$, we found that if the initial data $u_0<-1$ \textbf{but is not constant} on $\mathbb{S}^1$, there is an semi-infinite orbit initiating from $\Lambda_0$ and converge asymptotically to $\Lambda_-$. This phenomenon is detected neither by the mechanisms formulated in \cite{EP}, since in this case $(\Lambda_0,\Lambda_-)$ is not interlinked, nor by our results. So it is natural to ask

\begin{question}
Is there an abstract mechanism, for some suitable setting including Example \ref{contact Mobius} as a special case, for the existence of such orbits?
\end{question}

\section*{Acknowledgments}

All of the author are supported in part by the National Natural Science Foundation of China (Grant No. 12171096). L. Jin is also supported in part by the NSFC (Grant No. 11901293, 11971232). J. Yan is also supported in part by the NSFC (Grant No. 11790272). The first author would like to thank Dr. S.Suhr for his kind invitation and RUB (Ruhr-Universit\"{a}t Bochum) for its hospitality, where part of this work is done.


\end{document}